\documentclass{amsart}

\usepackage{amsfonts}
\usepackage{amssymb}
\usepackage{amsmath}
\usepackage{amsthm}
\usepackage{amsmath,amscd}
\usepackage[all]{xy}
\usepackage{enumerate}
\usepackage{url}

\usepackage[OT2,T1]{fontenc}
\DeclareSymbolFont{cyrletters}{OT2}{wncyr}{m}{n}
\DeclareMathSymbol{\Sha}{\mathalpha}{cyrletters}{"58}

\newtheorem{theorem}{Theorem}[section]
\newtheorem{lemma}[theorem]{Lemma}
\newtheorem{prop}[theorem]{Proposition}
\newtheorem{cor}[theorem]{Corollary}
\newtheorem{conjecture}[theorem]{Conjecture}

\theoremstyle{definition}
\newtheorem{definition}[theorem]{Definition}

\newtheorem{remark}[theorem]{Remark}

\newcommand{\Z}{\mathbb{Z}}

\newcommand{\Q}{\mathbb{Q}}

\newcommand{\F}{\mathbb{F}}

\newcommand{\bK}{\mathbb{K}}

\newcommand{\Kb}{\bar{K}}

\newcommand{\braces}[1]{\left\{#1\right\}}

\newcommand{\Sk}{\mathfrak{S}}

\newcommand{\mk}{\mathfrak{m}}

\newcommand{\Ac}{\mathcal{A}}

\newcommand{\Fc}{\mathcal{F}}

\newcommand{\Ic}{\mathcal{I}}

\newcommand{\Oc}{\mathcal{O}}

\newcommand{\Rc}{\mathcal{R}}
\newcommand{\Sc}{\mathcal{S}}

\newcommand{\Xc}{\mathcal{X}}

\newcommand{\chisub}[1]{{\chi_{\lower2.5pt\hbox{$\scriptstyle #1$}}}}

\newcommand{\beq}{\begin{equation}}
\newcommand{\eeq}{\end{equation}}

\newcommand{\beqa}{\begin{eqnarray}}
\newcommand{\eeqa}{\end{eqnarray}}

\newcommand{\beqaN}{\begin{eqnarray*}}
	\newcommand{\eeqaN}{\end{eqnarray*}}

\renewcommand{\phi}{\varphi}

\renewcommand{\mod}[1]{\hspace{.025in}\left(\text{mod } #1\right)}

\renewcommand{\dim}[1]{\text{dim}_{#1}}
\renewcommand{\ker}[1]{\text{ker}\left(#1\right)}
\newcommand{\rk}[1]{\text{rank}_{#1}}
\newcommand{\crk}[1]{\text{corank}_{#1}}

\newcommand{\isomto}{\stackrel{\sim}{\longrightarrow}}
\newcommand{\into}{\hookrightarrow}

\newcommand{\ip}[2]{\langle #1,#2 \rangle}

\newcommand{\Norm}[2]{N_{{#1}/{#2}}}
\newcommand{\Trace}[2]{\text{Tr}_{{#1}/{#2}}}

\newcommand{\muroots}{\ensuremath{\boldsymbol{\mu}}}
\newcommand{\pp}{\mathfrak{p}}
\newcommand{\qp}{\mathfrak{q}}
\newcommand{\Pp}{\mathfrak{P}}

\newcommand{\weilpairpol}[2]{\text{e}_{#1,\lambda}\left(#2\right)}

\newcommand{\invmap}[1]{\text{inv}_{#1}}

\newcommand{\Hom}[3]{\text{Hom}_{#1}(#2,#3)}
\newcommand{\End}[2]{\text{End}_{#1}(#2)}

\newcommand{\Gal}[2]{\text{Gal}({#1}/{#2})}

\newcommand{\Selmer}[3]{\text{Sel}_{{#1}}({#2}/{#3})}
\newcommand{\Sel}[3]{H^1_{{#1}}({#2},{#3})}
\newcommand{\TS}[2]{\Sha({#1}/{#2})}
\newcommand{\Coh}[3]{H^{#1}({#2},{#3})}
\newcommand{\Cocyc}[3]{Z^{#1}({#2},{#3})}
\newcommand{\Coch}[3]{C^{#1}({#2},{#3})}
\newcommand{\res}[1]{\text{res}_{#1}}

\numberwithin{equation}{section}

\makeatletter
\@namedef{subjclassname@2010}{%
  \textup{2010} Mathematics Subject Classification}
\makeatother

\begin{document}

\title[Local constants for abelian varieties]{Arithmetic local constants for abelian varieties with extra endomorphisms}

\author[Chetty]{Sunil Chetty}
\address{Mathematics Department\\ College of St. Benedict and St. John's University}
\email{schetty@csbsju.edu}

\date{\today}

\begin{abstract}
	This work generalizes the theory of arithmetic local constants, introduced by Mazur and Rubin, to better address abelian varieties with a larger endomorphism ring than $\Z$. We then study the growth of the $p^\infty$-Selmer rank of our abelian variety, and we address the problem of extending the results of Mazur and Rubin to dihedral towers $k\subset K\subset F$ in which $[F:K]$ is not a $p$-power extension.
\end{abstract}



\maketitle

\section{Introduction}
\label{intro}
In \cite{MR}, Mazur and Rubin introduce a theory of arithmetic local constants for an elliptic curve $E$ in terms of Selmer structures associated to $E$. With this theory they study, for an odd prime $p$, the growth in $\Z_p$-corank of the $p^\infty$-Selmer group $\Selmer{p^\infty}{E}{K}$ (see \S\ref{pSelcorank}) over a dihedral extension of number fields. To be precise, an extension $F/k$ is \emph{dihedral} if $k\subset K\subset F$ is a tower of number fields with $K/k$ quadratic, $F/k$ Galois, $F/K$ $p$-power abelian, and a lift of the non-trivial element $c\in\Gal{K}{k}$ acts on each $\sigma\in\Gal{F}{K}$ as $c\sigma c^{-1}=\sigma^{-1}$. They prove (under mild assumptions, see \cite[\S 7]{MR}) that the growth in the $\Z_p$-corank of $\Selmer{p^\infty}{E}{}$ over $F/K$ must be at least $[F:K]$.\\
\indent Here, we consider a more general context for the theory of local constants. In particular, we replace the elliptic curve $E/k$ with a pair $(X/k,\lambda)$ of an abelian variety $X/k$ and a polarization $\lambda:X\to X^{\vee}$ on $X$ of degree prime to $p$, where $X^{\vee}$ is the dual abelian variety. We consider the ring of integers $\Oc$ of a number field $\bK$, and assume $\Oc\subset\End{K}{X}$ is contained in the ring of endomorphisms of $X$ defined over $K$. The case $\Oc=\Z$ and $\bK=\Q$ is that of Mazur and Rubin in \cite{MR}. Recent work of Seveso \cite{Seveso} addresses similar questions for abelian varieties with real multiplication.\\
\indent The condition that $X$ has a polarization degree prime to $p$ implies that many of the constructions of \cite{MR} generalize verbatim\footnotemark\footnotetext{see subsection ``Generalizations'' in \cite[\S 1]{MR}}, with $E$ replaced by $X$. The goal in the present work is, in particular, to generalize Theorem 6.4 of \cite{MR} in the case that the endomorphism ring of $X$ is strictly larger than $\Z$.\\
\indent As a motivating example, consider $p$ an odd rational prime, $X=E$ an elliptic curve defined over $\Q$ with complex multiplication by the ring of integers $\Oc$ of a quadratic imaginary field $\bK$ in which $p$ does not split, and set $K=\bK$. The $\Z_p$-corank of $\Selmer{p^\infty}{E}{K}$ would be even, so $E$ would not satisfy the hypotheses of Theorem 7.2 of \cite{MR} and hence one does not obtain a lower bound for the $\Z_p$-corank of $\Selmer{p^\infty}{E}{F}$. One needs to consider $\Selmer{p^\infty}{E}{F}$ as a module over $\Oc\otimes\Z_p$ in order to obtain any useful generalization of the main tool (Theorem 6.4 of \cite{MR}) in the proof of Theorem 7.2 of \cite{MR}.

\subsection{Notation and Assumptions}
\label{notation}
Before continuing, we introduce some notation and assumptions that will be used until \S\ref{compositeresults}, where we will ease the restrictions on $F/K$.\\
%
\indent Fix an odd rational prime $p$. The tower $k\subset K\subset F$ is as above, with $K/k$ quadratic, $F/K$ an abelian $p$-extension, and $F/k$ dihedral. Also, $X/k$ and $\Oc\subset\End{K}{X}$ are as above, and we denote the cohomology groups $\Coh{i}{\Gal{\bar{K}}{K}}{X(\bar{K})}$ by $\Coh{i}{K}{X}$. Define a set $\Sk_F$ of primes $v$ of $K$ by
	$$\Sk_F:=\braces{\text{$v\mid p$, or $v$ ramifies in $F/K$, or where $X/K$ has bad reduction}},$$
and define $\Sk_L$ similarly for intermediate fields $K\subset L\subset F$. For a cyclic extension $L/K$ contained in $F$, define $A_L$ to be the twist of $X$, in the sense of \cite{MRS}, associated to $L/K$ (see \S\ref{SelmerTate} below).\\
\indent We assume that our prime $p$ is unramified in $\Oc\subset\End{K}{X}$ and we denote $\bK_\pp$ and $\Oc_\pp$ for the local field and ring, respectively, at a prime $\pp$ of $\Oc$ above $p$. For each prime $v$ of $K$ we fix an extension of $v$ to $\Kb$, which in turn fixes an embedding of $\Kb$ into an algebraic closure of $K_v$ and a decomposition subgroup $G_{K_v}=\Gal{\Kb_v}{K_v}\subset G_K$.\\
\indent We fix a polarization $\lambda:X\to X^{\vee}$ on $X$ of degree prime to $p$, thus fixing an isogeny $\lambda\in\Hom{}{X}{X^{\vee}}$ which has an inverse in $\lambda^{-1}\in\Hom{}{X^{\vee}}{X}\otimes\Q$. Associated to $\lambda$ is the Rosati involution on $\End{}{X}\otimes\Q$, given by
	$$\alpha\mapsto \alpha^{\dagger}:=\lambda^{-1}\circ\alpha^{\vee}\circ\lambda,$$
where $\alpha^\vee$ is the dual of $\alpha$. This in particular satisfies,
	$$\weilpairpol{\ell}{\alpha a,a'}=\weilpairpol{\ell}{a,\alpha^{\dagger}a'},$$
where $e_{\ell,\lambda}(\cdot,\cdot)=e_{\ell}(\cdot,\lambda(\cdot))$ is the Weil pairing and $a,a'\in T_{\ell}(X)\otimes\Q$ (see \cite[\S 16-17]{MiAV}).\\
\indent We assume that the non-trivial element $c\in\Gal{K}{k}$ acts as the Rosati involution on $\Oc\subset\End{K}{X}\otimes\Q$, and that $\Oc$ is taken to itself by the Rosati involution, i.e. $\Oc^c=\Oc^{\dagger}=\Oc.$
\begin{remark}
	Suppose $X=E$ is an elliptic curve defined over $k$ with complex multiplication by $\Oc\subset\bK$ and $\Oc\subset\End{K}{E}$. We know that the Rosati involution is the automorphism of $\Oc\otimes\Q=\bK$ of order 2. If $\bK\nsubseteq k$, then $k\bK=K$ and so the action of the Rosati involution and $c\in\Gal{K}{k}$ on $\Oc$ must coincide.
\end{remark}

\subsection{Main Results} 
\label{intromainresults}
With the above discussion in mind, the goal in the following is to keep track of the extra endomorphisms of the variety $X/k$. Effectively this amounts to extending the base ring (from $\Z_p$ to $\Oc\otimes\Z_p$) for the $p^\infty$-Selmer module, and as such the main results address this base extension.\\
\indent In \S\ref{torsionmodules} we address the important properties, for our purposes, of torsion $\Oc$-modules, noting Proposition \ref{skewsymmprop} for those modules equipped with a certain biliear form. In \S\ref{SelmerTate} we extend the results of \cite{MR} regarding Selmer structures and duality, and in \S\ref{Rrank} we apply those results to obtain information about the $\Oc/p\Oc$-rank of the relevant modules (as in \S 2 of \cite{MR}). This, in particular, motivates a generalized definition (in \S\ref{primeresults}) of the arithmetic local constant $\delta_v$, and combining \S\ref{torsionmodules}-\S\ref{Rrank} in \S\ref{pSelcorank} leads to our main result, Theorem \ref{genMR6.4}.\\
%
%
\indent As an application, in \S\ref{compositeresults} we are able to address another generalization mentioned in the introduction of \cite{MR}. In particular, we will consider dihedral towers $k\subset K\subset F$ where $[F:K]$ is not a prime power. For example, suppose $[F:K]$ is divisible by two distinct odd primes $p,q$ and $L/K$ is a cyclic extension contained in $F$. Then we have a $p$-power extension $M/K$ and a $q$-power extension $M'/K$ in $L$ (one of these may be trivial) such that $M\cap M'=K$ and $L=MM'$. We can apply Theorem \ref{genMR6.4} for $X$, $A_M$, and the ($p$-power) dihedral extension $M/k$ and then separately for $A_M$, $A_L$, and a ($q$-power) dihedral extension $M'/k$. Assuming Conjecture \ref{rankindepofl}, we can combine this information to compare $X$ and $A_L$.\\
\indent In addition to applications to growth in $p$-Selmer rank, it would be interesting to compare the individual $\delta_v$ to a quotient of the local root numbers for the $L$-function associated to $X$, as in \cite{ChCLC}. We leave this question to future work.

\section{Torsion $\Oc$-modules}
\label{torsionmodules}
In this section we consider various $\Oc$-modules, and so we prove some general results before applying them to our specific situation. Our abelian variety $X$ and the associated cohomology groups $\Coh{i}{K}{X}$ are the basic examples of $\Oc$-modules to keep in mind.\\
\indent As $\Oc/p\Oc$ may not be an integral domain, one does not have a natural definition of the $\Oc/p\Oc$-rank of an $\Oc/p\Oc$-module via its fraction field (since there would be no such field). However, since $p\Oc=\prod_i \pp_i$ with $\pp_i\neq\pp_j$ when $i\neq j$, one has 
	$$\Oc/p\Oc\cong\oplus_i(\Oc/\pp_i)$$ 
induced by the natural $\Oc\to\oplus_i(\Oc/\pp_i)$ maps.\footnotemark\footnotetext{Alternatively, one has $\Oc/p\Oc=\Oc\otimes_{\Z}(\Z/p\Z)$ and that $\Oc$ is a torsion-free, hence flat, $\Z$-module (see \cite[\S XVI.3]{LangAlg}), which yields the same decomposition.} Thus, $\Oc/p\Oc$ is a direct sum of fields $\Oc/\pp_i$, and each of these is a finite extension of $\F_p$.
\begin{definition}
\label{rankvector}
	Set $R=\Oc/p\Oc$ and $R_i=\Oc/\pp_i$, so $R\cong\oplus_{i=1}^m R_i$. For any $R$-module $M$ of finite type, define the \emph{$R$-rank} of $M$ to be
 		$$\rk{R}M:=(\ldots,\dim{R_i}M\otimes_{R}R_i,\ldots)\in\Z^m.$$
	We say $a=(a_1,\ldots,a_m)\in\Z^m$ is \emph{even} if $a_i$ is even for each $i$.
\end{definition}
A first, and most important, property of this definition of $R$-rank is that it behaves as one expects with respect to short exact sequences. We will exploit this property frequently. The proof of this and the subsequent Lemma are left as exercises for the reader.
\begin{prop}
	\label{rankexact}
	If $0\to M_1\to M_2\to M_3\to 0$ is a short exact sequence of $R$-modules then
	$$\rk{R}M_2=\rk{R}M_1+\rk{R}M_3.$$
\end{prop}
%
\begin{lemma}
\label{tensor-torsion}
	If $M$ is an $\Oc/p\Oc$-module of finite type (i.e. $M$ is $p$-torsion as an $\Oc$-module) then $M\otimes_R(\Oc/\pp)\cong M[\pp]$.
\end{lemma}
%
For any $R$-module $M$, we denote $M^\dagger$ for the $R$-module which has the same underlying set as $M$, but with $R$-action given by $rm:=r^\dagger m$. Also, for any abelian group $\Gamma$, we denote $\Hom{}{M}{\Gamma}:=\Hom{\Z}{M}{\Gamma}$ for the $R$-module of group homomorphisms from $M$ to $\Gamma$, with the $R$-action on $\Hom{}{M}{\Gamma}$ given by $(rf)(x)=f(rx)$.
\begin{lemma}
\label{rankM=rankMc}
 Suppose $M$ is an $R$-module and $\mathfrak{c}:M\stackrel{\sim}{\longrightarrow}M$ is an isomorphism of groups with $\mathfrak{c}(rm)=r^\dagger \mathfrak{c}(m)$. Then $M\cong M^\dagger$ as $R$-modules and in particular $\rk{R}M=\rk{R}M^\dagger.$
\end{lemma}
\begin{proof}
	The isomoprhism $\mathfrak{c}$ induces an $R$-isomorphism, since $\mathfrak{c}(rm)=r^\dagger \mathfrak{c}(m)$.
\end{proof}
\begin{lemma}
\label{R-Hom-ranks}
	$\rk{R}R_t=\rk{R}\Hom{}{R_t^\dagger}{\F_p}^\dagger$, for each $t\in\braces{1,\ldots,m}$.
\end{lemma}
\begin{proof}
	By Definition \ref{rankvector}, 
		$$\begin{array}{rl}
				\rk{R}R_t= & (\ldots,\dim{R_j}R_t\otimes_R R_j,\ldots)\\
				= & (0,\ldots,\dim{R_t} R_t,\ldots,0),\vspace{2mm}\\
				\rk{R}\Hom{}{R_t}{\F_p}^\dagger= & (\ldots,\dim{R_j}\Hom{}{R_t}{\F_p}^\dagger\otimes_R R_j,\ldots).
			\end{array}$$
	Since $\Hom{}{R}{\F_p}^\dagger$ is an $\Oc/p\Oc$-module, we can use Lemma \ref{tensor-torsion} to obtain
		\begin{equation}
		\label{tens-torHoms}	
			\dim{R_j}\Hom{}{R_t}{\F_p}^\dagger\otimes_R R_j= \dim{R_j}\Hom{}{R_t}{\F_p}^\dagger[\pp_j],
		\end{equation}
	and we claim that 
		\begin{equation}
		\label{relatingHoms}	
			\dim{R_j}\Hom{}{R_t}{\F_p}^\dagger[\pp_j]=
				\left\{
					\begin{array}{lr}
						0 & \text{~when~}R_t\neq R_j^\dagger\\
						1 & \text{~when~}R_t=R_j^\dagger\\	
					\end{array}
				\right\}
		\end{equation}
Consider $f\in\Hom{}{R_t}{\F_p}^\dagger[\pp_j]$, with $R_t\neq R_j^\dagger$. If $f(r^\dagger x)=0$ for all $x\in R_t$ and all $r^\dagger\in\pp_j^\dagger$ then $f$ is the zero map, since there exists some $r^\dagger\in\pp_j^\dagger$ such that $r^\dagger\not\in\pp_t$ and hence $r^\dagger R_t=R_t$. When $R_t=R_j^\dagger$, we have $r^\dagger x=0$ for all $x\in R_t$, and so $f(r^\dagger x)=0$ is satisfied for every $f\in\Hom{}{R_j^\dagger}{\F_p}^\dagger[\pp_j]$, and this set has $R_j$-dimension 1.\\
	\indent Now consider $R_t\neq R_s$. Then viewing $R_t\otimes_R R_s$ either as $R_t[\pp_s]$ or $R_s[\pp_t]$ shows that $R_t\otimes_R R_s$ is trivial, and hence has rank 0. When $R_t=R_s$, we have $R_t\otimes_R R_t=R_t$. From this and \eqref{relatingHoms}, we obtain $\dim{R}R_t=\dim{R}\Hom{}{R_t^\dagger}{\F_p}^\dagger$.
\end{proof}
\begin{remark}
	Alternatively, one can prove Lemma \ref{R-Hom-ranks} as follows. Define a perfect pairing $(~,~):R_t\times R_t^\dagger\to\F_p$ via $(x,y)\mapsto \Trace{R_t}{\F_p}(xy^\dagger)$. This pairing satisfies $(rx,y)=(x,r^\dagger y)$ and hence gives an $R_t$-module isomoprhism $R_t\cong \Hom{}{R_t^\dagger}{\F_p}^\dagger.$
\end{remark}
\begin{cor}
\label{M-Hom-ranks}
	If $M$ is an $R$-module of finite type, then $$\rk{R}M=\rk{R}\Hom{}{M^\dagger}{\F_p}^\dagger.$$
\end{cor}	
\begin{proof}
	This follows from the Lemma and $M\cong\oplus_{t} R_t^{n_t}$.
\end{proof}
The next proposition is analogous to a well-known theorem for alternating pairings on vector spaces. Specifically, if $k$ is a field with char$(k)\neq 2$ and there is a non-degenerate, skew-symmetric pairing on a finite dimensional $k$-vector space $V$, then $\dim{k} V$ is even (see \cite[\S XV.8]{LangAlg} or \cite[\S 9.5]{Rotman}).
\begin{prop}
\label{skewsymmprop}
	Suppose $A$ is a commutative ring, char$(A)\neq 2$, and $A\cong\oplus_{j=1}^n A_j$, where each $A_j$ is a local ring with principal maximal ideal $\mk_j$. Let $M$, $N$ be $A$-modules with $M$ finite and $[~,~]:M\times M\to N$ be a non-degenerate, skew-symmetric pairing which satisfies $[sx,y]=[x,sy]$ for all $x,y\in M$ and $s\in A$. Then there exist $A$-submodules $M'$, $M''$ with $M'\cong M''$ and $M\cong M'\oplus M''$.
\end{prop}
\begin{proof}
	\indent Let $M_j=M\otimes A_j$. We first note that $A\cong\oplus_j A_j$ implies $M\cong\oplus M_j$. Since $M$ is finite, we see that $x\in M_j$ implies $x\in M_j[\mk_j^t]$ for some $t$. For $i\neq j$, $x\in M_j$ and $y\in M_i$, we have that $[x,y]=0$. Indeed, there is some $\alpha\in\mk_j$ with $\alpha x=0$ which acts as a unit on $M_i$. Thus, there is some $y'\in M_i$ with $\alpha y'=y$ and so 
		$$0=[\alpha x,y']=[x,\alpha y']=[x,y].$$
	\indent Now, suppose $x\in M_j$ is of maximal order, i.e. that $x\in M_j[\mk_j^t]$ but $x\not\in M_j[\mk_j^{t-1}]$ and that $t$ is maximal. Let $\pi$ be a generator of $\mk_j$ in $A_j$. Since $\pi^{t-1}x\neq 0$ there is some $y\in M_j$ such that $[\pi^{t-1}x,y]\neq 0$. We then have $0\neq [\pi^{t-1}x,y]=[x,\pi^{t-1}y]$ and so $\pi^{t-1}y\neq 0$. In particular, this implies that $y\not\in M_j[\mk_j^{t-1}]$ and $y\in M_j[\mk_j^t]$, since $x$ was chosen to be of maximal order. Moreover, we have that $\text{span}_{A_j}\braces{x}\cong\text{span}_{A_j}\braces{y}$. We also note that if $w=ax$ for some $a\in A$ then 
		$$[x,w]=[x,ax]=[ax,x]=[w,x]$$
and so $[x,w]=0$.\\
\indent Set $U:=\text{span}_{A_j}\braces{x,y}$. We claim that $U\cap U^\bot=\braces{0}$. Let $z\in U\cap U^\bot$ with $z=ax+by$ for some $a,b\in A_j$, and suppose that $a\neq 0$. Since $A_j$ is a local ring, we have $\pi^t\nmid a$, so $a\mid \pi^{t-1}$. So, we can find $a'\in A_j$ such that $aa'=\pi^{t-1}$. Now, for $w=a'y$ we have
		$$\begin{array}{rll}
				[z,w]= & [ax+by,a'y] & =[ax,a'y]\\
				      = & [(aa')x,y] & =[\pi^{t-1}x,y]\neq 0,
			\end{array}$$
contradicting $z\in U^\bot.$ In the same way we can see that if $\pi^t\nmid b$ then we can find $w\in U$ such that $[z,w]\neq 0$.\\
\indent We are now left with the case that $\pi^t\mid a$ and $\pi^t\mid b$. Since $x$ was chosen to be of maximal order, this forces $z=0$ and it follows that $U\cap U^\bot=\braces{0}$. Also, the above argument shows that $U\cong A_j x\oplus A_j y$. The finiteness of $M$ (and hence $M_j$) then implies that we can decompose $M_j$ as $M_j=U\oplus U^\bot$ and by induction we obtain the claim.
\end{proof}
\begin{remark}
\label{skewsymmpropremark}	
	Recall that $R=\Oc/p\Oc$, $R_j=\Oc/\pp_j$, and set $S=\Oc\otimes\Z_p$ and $S_j=\Oc_{\pp_j}$. We have decompositions $R\cong\oplus_j R_j$ and $S\cong\oplus_j S_j$. For $\Rc_L:=R_L\otimes\Z_p$, where $R_L$ is as in \S 3 of \cite{MR} (see also \S 3 below), we again have a decomposition $\Oc\otimes\Rc_L\cong\oplus_j(\Oc_{\pp_j}\otimes R_L)$. In what follows, these rings will play the role of $A$ in the above proposition. 
\end{remark}
\section{Selmer Structures and Tate Duality}
\label{SelmerTate}
As our goal is to establish a theorem analogous to Theorem 6.4 of \cite{MR}, we need to generalize the results of \cite{MR} regarding the pairing of Tate's local duality in order to yield information about the Selmer structures of Definition \ref{genSelmer-map} as $\Oc$-modules.\\
\indent Using Definition 3.3 of \cite{MR} (see also Definition 1.1 of \cite{MRS}), we have the $\Ic$-twist $A$ of $X$ exactly as in the elliptic curve case $X=E$. Specifically, for a cyclic extension $L/K$ contained in $F$, let $\rho_L$ denote the unique faithful irreducible rational representation of $\Gal{L}{K}$. Define the $\Ic_L$-twist of $X$ to be $A_L:=\Ic_L\otimes X$, where 
	$$\Ic_L:=\Q[\Gal{F}{K}]_L\cap\Z[\Gal{F}{K}]$$
and $\Q[\Gal{F}{K}]_L$ is the sum of all (left) ideals of $\Q[\Gal{F}{K}]$ isomorphic to $\rho_L$. We define the ring $R_L$ (mentioned in Remark \ref{skewsymmpropremark}) as the maximal order of $\Q[\Gal{F}{K}]_L$, and when $[L:K]=p^m$ we have that $R_L\cong\Z[\mu_{p^m}]$ has a unique prime above $p$.
\begin{remark}
\label{O-twist-vs-Z-twist}
	By definition (in \cite{MRS}), when $\Ic_L$ is a $\Z$-module, the twist $A_L=\Ic_L\otimes X$ is a $\Z$-module. However, we may regard it as an $\Oc$-module, simply by letting $\Oc$ act on $\Ic_L\otimes X$ via its action on $X$. The resulting module coincides with the $\Oc$-module $\Ic'_L\otimes X$ obtained by twisting $X$ with the $\Oc$-module
	$$\Ic'_L:=\bK[\Gal{F}{K}]_L\cap\Oc[\Gal{F}{K}].$$
\end{remark}
\begin{prop}
\label{gen4.1}
	For $\hat{p}$ the unique prime above $p$ in $\Ic_L$, there is a canonical $\Gal{\Kb}{K}$-isomorphism $A_L[\hat{p}]\cong X[p]$.
\end{prop}
\begin{proof}
	This is exactly as in Proposition 4.1 of \cite{MR} (also Remark 4.2 in \cite{MR}), where our $\hat{p}$ is their $\pp=\pp_L$.
\end{proof}
We are concerned with the following Selmer structures, analogous to those of \S 2 and \S 4 of \cite{MR}.
\begin{definition} 
\label{genSelmer-map}
Define a Selmer structure $\Xc$ on $X[p]$ as the collection of $\Oc$-modules $\Sel{\Xc}{K_v}{X[p]}$, defined to be, for each $v$, the image of 
	$$X(K_v)/pX(K_v)\into\Coh{1}{K_v}{X[p]}.$$
Fix a generator $\pi$ of $\hat{p}$, with $\hat{p}$ as in Proposition \ref{gen4.1}. Define a Selmer structure $\Ac$ on $X[p]$ by setting, for each $v$, $\Sel{\Ac}{K_v}{X[p]}$ to be the image of 
	$$A_L(K_v)/\pi A_L(K_v)\into\Coh{1}{K_v}{A_L[\hat{p}]}\cong\Coh{1}{K_v}{X[p]}.$$
We note that the image in $\Coh{1}{K_v}{X[p]}$ is independent of the choice of our generator. As in \cite[\S 1]{MR}, define
	$$\begin{array}{rl}
			\vspace{.08in}\Sel{\Xc+\Ac}{K_v}{X[p]}:=\Sel{\Xc}{K_v}{X[p]}+\Sel{\Ac}{K_v}{X[p]}\\			
			\Sel{\Xc\cap\Ac}{K_v}{X[p]}:=\Sel{\Xc}{K_v}{X[p]}\cap\Sel{\Ac}{K_v}{X[p]}.
		\end{array}$$
\end{definition}
\begin{definition}
\label{self-dual}
We say that a Selmer structure $\Fc$ on $X[p]$ is \emph{self-dual} if for every prime $v$, $\Sel{\Fc}{K_v}{[X[p]}$ is its own orthogonal complement under the pairing of Tate's local duality:
	\begin{equation}
	\label{Tatepair}
		\ip{~}{~}_v:\Coh{1}{K_v}{X[p]}\times\Coh{1}{K_v}{X[p]}\to\Coh{2}{K_v}{\muroots_p}=\F_p.
	\end{equation}
\end{definition}
We note that in Definition \ref{self-dual}, we are making use of our assumption that $X$ has a polarization of degree prime to $p$ in order to have \eqref{Tatepair} as a \emph{self}-pairing.
\begin{definition}
\label{defselgrp}
Given a Selmer structure $\Fc$ on $X[p]$, define the \emph{Selmer group} to be
	$$\Sel{\Fc}{K}{X[p]}:=\ker{\Coh{1}{K}{X[p]}\to\prod_v\Coh{1}{K_v}{X[p]}/\Sel{\Fc}{K_v}{X[p]}}.$$
Thus, $\Sel{\Fc}{K}{X[p]}$ is the set of classes whose localizations are in $\Sel{\Fc}{K_v}{X[p]}$, or in other words the classes satisfying the local conditions defined by $\Fc$.
\end{definition}
\begin{prop}
\label{MR2.1selfdual}
	The Selmer structures $\Xc$ and $\Ac$ on $X[p]$ are self-dual.
\end{prop}
\begin{proof}
	The Tate pairing is the same as that in \cite{MR}, and Tate local duality holds for a general abelian variety (see \cite[\S 1.4]{MR}). This shows that $\Xc$ is self-dual. For $\Ac$, the proof is exactly Proposition A.7 of Appendix A of \cite{MR}, noting that we need only regard $A_L$ as a $\Z$-module here.
\end{proof}	
The pairing \eqref{Tatepair} is not $\Oc$-linear, but understanding the interplay of the pairing and the map induced by $c$ on the local cohomology groups $\Coh{1}{K_v}{X[p]}$ provides information (see Lemma \ref{gen1.3ii} below) about the $R$-rank of certain Selmer groups.\\
\indent Now, we fix a lift of the nontrivial element $c\in\Gal{K}{k}$ to $\Gal{\Kb}{k}$, which we also denote $c$. As $c\in G_k$ with $c(K)=K$, we have that $c:K_v\isomto K_{v^c}$. The maps $c:G_K\to G_K:~s\mapsto c^{-1}sc$ and $c:M\to M:~a\mapsto c(a)$, for any $G_k$-module $M$, are compatible in the sense of \cite[\S VII.5]{SerreLF}, and hence induce $c^*:\Coh{*}{K}{M}\to\Coh{*}{K}{M}$ on cohomology. Similarly, from $c:G_{K_v}\to G_{K_{v^c}}$ we obtain $c^*:\Coh{*}{G_{K_{v^c}}}{M}\to\Coh{*}{G_{K_v}}{M}$.
\begin{lemma}
\label{independentoflift}
	For $G_k$-module $M$, the map $c^*:\Coh{1}{G_K}{M}\to\Coh{1}{G_K}{M}$ induced by the lift $c\in G_k$ of $c$ is independent of the choice of lift.
\end{lemma}
\begin{proof}
	The claim follows from a special case of Proposition 3 of \S VII.5 of \cite{SerreLF}.
\end{proof}
\begin{lemma}
\label{galeqweil}
	Let $M$ and $N$ be two $G_k$-modules and $\phi:M\to N$ a $G_k$-equivariant map. Then for the map $\phi^*:\Coh{*}{K}{M}\to\Coh{*}{K}{N}$ induced by $\phi$,
		$$\phi^*\circ c^*=c^*\circ \phi^*:\Coh{*}{K_{v^c}}{M}\to\Coh{*}{K_v}{N}.$$
\end{lemma}
\begin{proof}
	Let $G=\Gal{\Kb_v}{K_v}$ and $G'=\Gal{\Kb_{v^c}}{K_{v^c}}$. We prove the claim on cochains. For each $i\geq 0$, let $P_i:=\Z[G^{i+1}]$, be the free module generated by elements $(g_0,\ldots,g_i)\in G^{i+1}$, with a $G$-action by 
	$$s.(g_0,\ldots,g_i)=(s.g_0,\ldots,s.g_i).$$
These form the standard resolution for $\Z$ (see \cite[\S VII.3]{SerreLF} or \cite[\S I.5]{Brown}).\\
	\indent Suppose $f\in\Hom{G'}{P'_i}{M}$. Then
		$$\begin{array}{rl}
				c^*(f)(g_0,\ldots,g_i)= & c(f(c^{-1}g_0c,\ldots,c^{-1}g_ic))\\
				\phi^*(f)(g_0,\ldots,g_i)= & \phi(f(g_0,\ldots,g_i)),\\
			\end{array}$$
and it follows that
	$$(\phi^*\circ c^*)(f)(g_0,\ldots,g_i)=(c^*\circ \phi^*)(f)(g_0,\ldots,g_i),$$
using the $G_k$-equivariance of $\phi$.
\end{proof}
Let $W=X[p]$. Denote $e^*:\Coh{*}{K}{W\otimes W}\to\Coh{*}{K}{\muroots_p}$ for the map induced by the Weil pairing $e_{p,\lambda}$ on $W$. We will also use $e^*$ for the maps induced by $e_{p,\lambda}$ on $G_{K_v}$-cohomology and $G_{K_{v^c}}$-cohomology, and context will make the notation clear. We know that $e_{p,\lambda}$ is $\Gal{\Kb}{k}$-equivariant (see \cite[\S III.8]{Silv} or \cite[\S 12]{MiAV}). By Lemma \ref{galeqweil}, we see that
	$$e^*\circ c^*=c^*\circ e^*:\Coh{*}{K_{v^c}}{W\otimes W}\to\Coh{*}{K_v}{\muroots_p}.$$
\begin{prop}
\label{Tatepairrelation}  
  Suppose $S$ is a finite set of primes $v$ of $K$ such that $v\in S$ if and only if $v^c\in S$. For any $a$, $b\in\oplus_{v\in S}\Coh{1}{K_v}{W}$, let $\ip{a}{b}:=\sum_{v\in S}\ip{a_v}{b_v}_v$. Then
  	$$\ip{a}{c^*(b)}=\ip{c^*(a)}{b}.$$
\end{prop}
\begin{proof}
	Recall that $\ip{~}{~}_v$ is defined via the composition (cf. \cite[\S 1.4]{Rubin})
	$$\xymatrix@R=25pt
			{\Coh{1}{K_v}{W}\otimes\Coh{1}{K_v}{W}\ar[d]^-\cup & & \\
			 \Coh{2}{K_v}{W\otimes W}\ar[r]^-{e^*} & \Coh{2}{K_v}{\muroots_p}\ar[r]^-{\invmap{v}} & \muroots_p.}$$
The cup product $\cup$ is functorial, so the commutative diagram
 $$\xymatrix@R=25pt@C=1pt
 			{\Coh{1}{K_{v^c}}{W}\ar[d]^-{c^*} & \otimes & \Coh{1}{K_{v^c}}{W}\ar[d]^-{c^*}\ar[rrrrr]^-\cup 
 					&&&&& \Coh{2}{K_{v^c}}{W\otimes W}\ar[d]^-{c^*}\\ 
 			 \Coh{1}{K_v}{W} & \otimes & \Coh{1}{K_v}{W}\ar[rrrrr]^-\cup 
 			 		&&&&& \Coh{2}{K_v}{W\otimes W}\\
 			}$$
implies $a\cup c^*(b)=c^*c^*(a)\cup c^*(b)=c^*(c^*(a)\cup b)$. Also we can see that, for all $i\geq 0$,
	$$\xymatrix@R=30pt@C=30pt
		 { \Coh{i}{K}{W}\ar[r]^{\sim}_{c^*}\ar[d]^{{\text{res}}_{v^c}} & \Coh{i}{K}{W}\ar[d]^{{\text{res}}_{v}}\\
				\Coh{i}{K_{v^c}}{W}\ar[r]^{\sim}_{c^*} & \Coh{i}{K_v}{W}.\\
		 }$$
commutes by recalling that on cochains $\res{v}(f)$ is restriction of the map $f$. Using Lemma \ref{galeqweil} and the property $\invmap{v}\circ c^*=\invmap{v^c}$ (see \cite[\S\S XI.1-XI.2]{SerreLF}, particularly Proposition 1) of the local invariant map, we see $\ip{a}{c^*(b)}=\ip{c^*(a)}{b}.$
\end{proof}
The next proposition shows how the $R$-action on our cohomology groups interacts with the pairing \eqref{Tatepair}.
\begin{prop}
\label{Tatepairadjoint}
	For any $a$, $b\in\Coh{1}{K_v}{X[p]}$ and $r\in R$, $\ip{ra}{b}_v=\ip{a}{r^\dagger b}_v$.
\end{prop}
\begin{proof}
	Let $W=X[p]$ as above, and let $x,y\in W$ and $r\in\Oc$. The claim is a consequence of the identity $\weilpairpol{p}{rx,y}=\weilpairpol{p}{x,r^\dagger y}$. As $e_{p,\lambda}$ is bilinear, it can be viewed as a map on $W\otimes_{\Z}W$, and the above property becomes $\weilpairpol{p}{rx\otimes y}=\weilpairpol{p}{x\otimes r^\dagger y}$. Now, for $a,b\in\Coh{1}{K_v}{W}$ we have $r$ and $r^\dagger$ acting by $(ra)(g)=r.a(g)$ and $(r^\dagger b)(g)=r^\dagger .b(g).$ Thus, keeping in mind that $\Oc\subset\End{K}{X}$, it follows that
		$$\begin{array}{rl}
				e^*((ra)\cup b)(g,h)= & \weilpairpol{p}{((ra)\cup b)(g,h)}\\
				= & \weilpairpol{p}{(a\cup (r^\dagger b))(g,h)}\\
				= & e_{p,\lambda}^*((a\cup (r^\dagger b))(g,h),
			\end{array}$$
and so
		$$\begin{array}{rl}
				\ip{ra}{b}_v= & \invmap{v}\circ e_{p,\lambda}^*((ra)\cup b)\\
				= & \invmap{v}\circ e_{p,\lambda}^*(a\cup (r^\dagger b))=\ip{a}{r^\dagger b}_v.\\
			\end{array}$$
\end{proof}
\begin{cor}
\label{Tateorth}
	The orthogonal complement of $\Coh{1}{K_v}{X[p]}[\pp]$ under \eqref{Tatepair} is $\oplus_{\qp\neq\pp^\dagger}\Coh{1}{K_v}{X[p]}[\qp]$.
\end{cor}
\begin{proof}
	Set $M=\Coh{1}{K_v}{X[p]}$. Let $a\in M[\pp]$, $b\in M$, and $r\in\pp$. Then
	$$0=\ip{0}{b}_v=\ip{ra}{b}_v=\ip{a}{r^\dagger b}_v,$$
so $r^\dagger M\subset M[\pp]^\bot$ and in turn $\pp^\dagger M\subset M[\pp]^\bot$. Since $M=\oplus_{\qp\mid p} M[\qp]$, we see that $\pp^\dagger M=\oplus_{\qp\neq\pp^\dagger}M[\qp]\subset M[\pp]^\bot,$ and non-degeneracy finishes the claim.
\end{proof}

\section{$\Oc/p\Oc$-rank}
\label{Rrank}
Recall $\Sk_L$ is a finite set of primes of $K$ containing those which divide $p$ or are ramified in $L/K$ or where $X$ does not have good reduction. In this section we fix a cyclic extension $L/K$ contained in $F$.
\begin{lemma}
\label{X=A}
	For $v\not\in\Sk_L$, the Selmer structures $\Xc$ and $\Ac$ on $X[p]$ coincide.
\end{lemma}
\begin{proof}
This is Corollary 4.6 of \cite{MR}, which uses Lemma 19.3 of \cite{Cassels}. Specifically, both $\Xc$ and $\Ac$ are self-dual (cf \S\ref{SelmerTate}) and when $v\not\in\Sk_L$ then both $T_{p}(X)$ and $T_{p}(A_L)$ are unramified at $v$. Thus,
		$$\Sel{\Xc}{K_v}{X[p]}=\Sel{\Ac}{K_v}{X[p]}=\Coh{1}{K_v^{ur}/K_v}{X[p]}.$$
\end{proof}
Let $R=\Oc/p\Oc$ and $R_i=\Oc/\pp_i$ be as in the previous section. We now generalize the main results of \S 1 of \cite{MR} regarding self-dual Selmer structures. Later, determining the difference in the $(\Oc\otimes\Z_p)$-corank of the $p^\infty$-Selmer groups associated to $\Xc$ and $\Ac$ will be reduced to determining the difference in the $R$-corank of the $p$-Selmer groups, and Theorem \ref{gen1.4} below describes the latter. We phrase the result specifically in terms of the Selmer structures $\Xc$ and $\Ac$, as we make use of the assumption on $c$ introduced in the beginning of \S\ref{intro} to prove Lemma \ref{gen1.3i}.
\begin{remark}
\label{rankB=rankBc}
	The following is an example of an application of Lemma \ref{rankM=rankMc}.	Set $W=X[p]$ and 
		$$B=\bigoplus_{v\in\Sk_L}(\Sel{\Xc+\Ac}{K_v}{W}/\Sel{\Xc\cap\Ac}{K_v}{W}).$$
We check that $v\in\Sk_L$ if and only if $v^c\in\Sk_L$. Since $c\in\Gal{K}{k}$, we have $v\mid p$ implies $v^c\mid p$. Also, if $w$ witnesses that $v$ is ramified in $L/K$ then $w^c$ witnesses that $v^c$ is ramified in $L/K$. Lastly, since $X$ is defined over $k$, $X$ has good reduction at $v$ if and only if $X$ has good reduction at $v^c$.\\
	\indent The automorphism $c$ induces an isomorphism $X(K_v)\isomto X(K_{v^c})$ and in turn $\Coh{1}{K_v}{W}\isomto\Coh{1}{K_{v^c}}{W}$. This restricts to a group isomorphism
		$$\Sel{\Xc}{K_v}{W}\isomto\Sel{\Xc}{K_{v^c}}{W}.$$
We have analogous isomorphisms for $\Sel{\Ac}{K_v}{W}$. As $B$ is a direct sum taken over all $v\in\Sk_L$, we know that $\Sel{\Xc+\Ac}{K_v}{W}$ and $\Sel{\Xc+\Ac}{K_{v^c}}{W}$ occur symmetrically in $B$. Thus, 
	$$\begin{array}{rl}
			B= & \bigoplus_{v\in\Sk_L}(\Sel{\Xc+\Ac}{K_v}{W}/\Sel{\Xc\cap\Ac}{K_v}{W})\\
			 \cong & \bigoplus_{v^c\in\Sk_L}(\Sel{\Xc+\Ac}{K_{v^c}}{W}/\Sel{\Xc\cap\Ac}{K_{v^c}}{W})=B
		\end{array}$$
and so $c:B\isomto B$. Lemma \ref{rankM=rankMc} then gives $\rk{R}B=\rk{R}B^\dagger$.
\end{remark}
Recall the definition of a Selmer group, e.g. $\Sel{\Xc}{K}{X[p]}$, in Definition \ref{defselgrp}. The following Lemmas generalize Proposition 1.3 of \cite{MR}.
\begin{lemma}
	\label{gen1.3i}
	Since $\Xc$ and $\Ac$ are self-dual,\\
	\indent $\rk{R}\Sel{\Xc+\Ac}{K}{X[p]}/\Sel{\Xc\cap\Ac}{K}{X[p]}$
	\begin{flushright}
		$=\sum_{v\in S}\rk{R}(\Sel{\Xc}{K_v}{X[p]}/\Sel{\Xc\cap\Ac}{K_v}{X[p]}).$
	\end{flushright}
\end{lemma}
\begin{proof}	
	We follow the ideas of Proposition 1.3 of \cite{MR}, noting the adjustments needed to address $R$-rank. Let $W$ and $B$ be as in Remark \ref{rankB=rankBc}. The Tate pairing restricts to $\Sel{\Xc+\Ac}{K_v}{W}$ for each $v$, and since $\Xc$ and $\Ac$ are self-dual we obtain a pairing $\ip{~}{~}:B\times B\to\F_p$.\\
	\indent Defining $C_{\Xc}$ (resp. $C_{\Ac}$) to be the projection of $\oplus_{v}\Sel{\Xc}{K_v}{W}$ (resp. $\oplus_{v}\Sel{\Ac}{K_v}{W}$) in $B$, the local self-duality of $\Xc$ (resp. $\Ac$) implies that $C_\Xc$ (resp. $C_\Ac$) is its own orthogonal complement under $\ip{~}{~}$. Using these orthogonality relations, we will show
	\begin{equation}
	\label{BCC_X}	
	\rk{R}C=\rk{R}C_{\Xc}=\rk{R}C_{\Ac}=\frac{1}{2}\rk{R}B.
	\end{equation}
	First we note $B=C_{\Xc}\oplus C_{\Ac}$, and since $C_{\Xc}^\bot=C_{\Xc}$ and $C_{\Ac}^\bot=C_{\Ac}$, the pairing $\ip{~}{~}$ restricts to a non-degenerate pairing on $C_{\Xc}\times C_{\Ac}$. From this we obtain in the usual way (see \cite[\S I.9]{LangAlg} or \cite[\S XIII.5]{LangAlg}) an $R$-isomorphism $C_{\Xc}\to\Hom{}{C_{\Ac}}{\F_p}^\dagger$ which implies
	$$\rk{R}C_{\Xc}=\rk{R}\Hom{}{C_{\Ac}}{\F_p}^\dagger=\rk{R}C_{\Ac}^\dagger,$$
	using Corollary \ref{M-Hom-ranks} for the right-hand equality. Then by Lemma \ref{rankM=rankMc}, as in Remark \ref{rankB=rankBc}, we see
	$$\rk{R}C_{\Xc}=\rk{R}C^\dagger_{\Ac}=\rk{R}C_{\Ac}.$$
	Thus, we have the middle and right-hand equalities of \eqref{BCC_X}.\\
	\indent Similarly, from $B\times B\to \F_p$ and $C=C^\bot$, we obtain $C\times(B/C)\to\F_p$ which gives
	$$\rk{R}C=\rk{R}\Hom{}{B/C}{\F_p}^\dagger=\rk{R}(B/C)^\dagger,$$
	and in turn, again by Lemma \ref{rankM=rankMc}, we have $\rk{R}C=\rk{R} B/C$. Now using the exact sequence (of $R$-modules)
	$$0\to C\to B\to B/C\to 0$$
	and Proposition \ref{rankexact}, we have
	$$\rk{R}B=\rk{R}C+\rk{R}(B/C)=2\rk{R}C,$$
	and hence the left-hand equality of \eqref{BCC_X}. The result now follows from
	$$C\cong\Sel{\Xc+\Ac}{K_v}{W}/\Sel{\Xc\cap\Ac}{K_v}{W}
	~\text{and}~
	C_{\Xc}\cong\oplus_v\Sel{\Xc}{K_v}{W}/\Sel{\Xc\cap\Ac}{K_v}{W}.$$
\end{proof}
\begin{lemma}
	\label{gen1.3ii}
	With the same assumptions and notation of Lemma \ref{gen1.3i},
	$$\rk{R}\Sel{\Xc+\Ac}{K}{W}\equiv\rk{R}(\Sel{\Xc}{K}{W}+\Sel{\Ac}{K}{W})\mod{2}.$$
\end{lemma}
\begin{proof}
	Again, we follow Proposition 1.3 of \cite{MR}. For $u\in\Sel{\Xc+\Ac}{K}{W}$, write $u_s\in C$ for the localization of $u$, and $u_x$, $u_a$ for the projections of $u_s$ to $C_{\Xc}$, $C_{\Ac}$, respectively. Using the symmetry of $\ip{~}{~}$, the pairing
	$$[~,~]:\Sel{\Xc+\Ac}{K}{W}\times\Sel{\Xc+\Ac}{K}{W}\to\F_p:~[u,w]:=\ip{u_x}{w_a}$$
	is skew-symmetric. Also, exactly as in \cite{MR}, the kernel of $[~,~]$ is exactly $\Sel{\Xc}{K}{W}+\Sel{\Ac}{K}{W}$, and so $[~,~]$ induces an $\F_p$-valued, non-degenerate, skew-symmetric pairing on
	$$H:=\Sel{\Xc+\Ac}{K}{W}/(\Sel{\Xc}{K}{W}+\Sel{\Ac}{K}{W}).$$
	Since $[~,~]$ is defined in terms of $\sum_{v\in\Sk_L}\ip{~}{~}_v$, we use Propositions \ref{Tatepairadjoint} and \ref{Tatepairrelation}, respectively, to see that
	$$[u,rw]=[r^\dagger u,w]~~~\text{and}~~~[u,c^*(w)]=[c^*(u),w].$$
	Define $[~,~]'$ on $H$ by $[u,w]':=[u,c^*(w)]$. The non-degeneracy and skew-symmetry of $[~,~]$ imply that $[~,~]'$ is non-degenerate and skew-symmetric also. In addition, the two properties above imply that $[ru,w]'=[u,rw]'$ and with this pairing Proposition \ref{skewsymmprop} (with $A=R$) shows that $\rk{R}H$ is even.
\end{proof}
\begin{theorem}
\label{gen1.4}
	Since $\Xc$ and $\Ac$ are self-dual,\\
		\indent $\rk{R}\Sel{\Xc}{K}{X[p]}-\rk{R}\Sel{\Ac}{K}{X[p]}$
		\begin{flushright}
			$\equiv\sum_{v\in S}\rk{R}(\Sel{\Xc}{K_v}{X[p]}/\Sel{\Xc\cap\Ac}{K_v}{X[p]})\mod{2}.$\\
		\end{flushright}
\end{theorem}
\begin{proof}
Applying Lemma \ref{X=A}, the claim follows from the congruences\\
		\indent $\rk{R}\Sel{\Xc}{K}{X[p]}-\rk{R}\Sel{\Ac}{K}{X[p]}$
			\begin{flushright}
				$\begin{array}{rl}
					\equiv\vspace{0.05in} & \rk{R}\Sel{\Xc}{K}{X[p]}+\rk{R}\Sel{\Ac}{K}{X[p]}\\
					\equiv\vspace{0.05in} & \rk{R}(\Sel{\Xc}{K}{X[p]}+\Sel{\Ac}{K}{X[p]})+\rk{R}\Sel{\Xc\cap\Ac}{K}{X[p]}\\
					\equiv\vspace{0.05in} & \rk{R}\Sel{\Xc+\Ac}{K}{X[p]}-\rk{R}\Sel{\Xc\cap\Ac}{K}{X[p]}\\
					\equiv & \sum_{v\in S}\rk{R}(\Sel{\Xc}{K_v}{X[p]}/\dim{}\Sel{\Xc\cap\Ac}{K_v}{X[p]})\mod{2}.
				 \end{array}$\\
			\end{flushright}
	The last two steps follow from Lemmas \ref{gen1.3i} and \ref{gen1.3ii}.
\end{proof}
\begin{remark}
\label{deflocalconstremark}
	The summands in the right-hand side of Theorem \ref{gen1.4} motivate Definition \ref{genMR4.5} below of the arithmetic local constants $\delta_v$.
\end{remark}

\section{$p$-Selmer corank}
\label{pSelcorank}
The $p$-Selmer group $\Sel{\Xc}{K}{X[p]}=\Selmer{p}{X}{K}$ sits in the exact sequence (see for example \cite[\S X.4]{Silv})
	\begin{equation}
	\label{pselmerexactsequence}		
			0\to X(K)\otimes\Z/p^m\Z\to \Selmer{p^m}{X}{K}\to \Sha(X/K)[p^m]\to 0
	\end{equation}
and passing to the limit $\Selmer{p^\infty}{X}{K}$ sits in
	\begin{equation}
	\label{limitselmerexactsequence}
			0\to X(K)\otimes\Q_p/\Z_p\to \Selmer{p^\infty}{X}{K}\to \Sha(X/K)[p^\infty]\to 0.
	\end{equation}
We have similar sequences for $\Sel{\Ac}{K}{X[p]}=\Selmer{\hat{p}}{A_L}{K}$ and for the associated direct limit $\Selmer{p^\infty}{A_L}{K}$.\\
\indent We next generalize Proposition 2.1 of \cite{MR}, but in order to do so we need to define a notion of corank over the ring $\Oc\otimes\Z_p$ (particularly in the case that it is not an integral domain). Again, we have a decomposition
	$$\Oc\otimes\Z_p\cong\oplus_i \Oc_{\pp_i}.$$
\begin{definition}
\label{rankvector2}
	Let $S:=\Oc\otimes\Z_p$ and $S_i:=\Oc_{\pp_i}$. For an $S$-module $M$, define the $S$-corank of $M$ to be
		$$\crk{S}M:=(\ldots,\crk{S_i}M\otimes\Oc_{\pp_i},\ldots).$$
\end{definition} 
\begin{prop}
\label{MR2.1congruence}
		$$\crk{S}\Selmer{p^\infty}{X}{K}\equiv\rk{R}\Selmer{p}{X}{K}-\rk{R}X(K)[p]\mod{2}.$$
\end{prop}
\begin{proof}
	We follow the strategy of Proposition 2.1 of \cite{MR}. Let
		$$\begin{array}{rl}
				\vspace{0.05in}d:= & \rk{R}(\Selmer{p^\infty}{X}{K}/\Selmer{p^\infty}{X}{K}_{\text{div}})[p]\\
        	= & \rk{R}(\TS{X}{K}[p^\infty]/\TS{X}{K}[p^\infty]_{\text{div}})[p].
      \end{array}$$
  We have
  	$$\begin{array}{rl}
  			\vspace{0.05in}\crk{S}\Selmer{p^\infty}{X}{K}= & (\ldots,\crk{S_i}\Selmer{p^\infty}{X}{K}\otimes S_i,\ldots)\\
  			\vspace{0.05in}= & (\ldots,\rk{R_i}\Selmer{p^\infty}{X}{K}_{\text{div}}[p]\otimes R_i,\ldots)\\
  			\vspace{0.05in}= & (\ldots,\rk{R_i}\Selmer{p^\infty}{X}{K}[p]\otimes R_i,\ldots)-d\\
  			= & \rk{R}\Selmer{p^\infty}{X}{K}[p]-d,
  		\end{array}$$
  with the first and last equalities by definition, and the others as in \cite{MR}. From \eqref{limitselmerexactsequence} we obtain another sequence
  	$$0\to (X(K)\otimes\Q_p/\Z_p)[p]\to \Selmer{p^\infty}{X}{K}[p]\to \Sha(X/K)[p]\to 0$$
  and then applying Proposition \ref{rankexact} we have
  	$$\rk{R}\Selmer{p^\infty}{X}{K}[p]=\rk{R}(X(K)\otimes\Q_p/\Z_p)[p]+\rk{R}\Sha(X/K)[p].$$
  From \eqref{pselmerexactsequence} and Proposition \ref{rankexact} we obtain
  	$$\rk{R}\Selmer{p}{X}{K}=\rk{R}(X(K)/pX(K))+\rk{R}\Sha(X/K)[p].$$
  Combining these, we see that\\
  	\vspace{-.05in}
  	\indent $\crk{S}\Selmer{p^\infty}{X}{K}-\rk{R}\Selmer{p}{X}{K}$
  	\begin{center}
  		$\begin{array}{rl}
  			\vspace{0.05in}= & \rk{R}\Selmer{p^\infty}{X}{K}[p]-d-\rk{R}\Selmer{p}{X}{K}\\
  			\vspace{0.05in}= & \rk{R}(X(K)\otimes\Q_p/\Z_p)[p]-\rk{R}(X(K)/pX(K))-d\\
  			= & -\rk{R}X(K)[p]-d.\\
  		 \end{array}$
  	\end{center}
  	\vspace{-.05in}
  Here we have cancelled the $\Sha(X/K)[p]$ terms in the second equality, and the last equality follows from the exact sequence
  	$$0\to X(K)[p]\to X(K)\otimes\Z/p\Z\to (X(K)\otimes\Q_p/\Z_p)[p]\to 0,$$
defined by considering each term as an $\Oc$-module and decomposing each term as in \cite[\S 11.2]{Rotman}, and applying \cite[\S XVI.2]{LangAlg}.\\
\indent It remains to see that $d$ is even, which will show that the above equality implies the desired congruence $\mod{2}$. We prove $d$ is even below in Proposition \ref{Sha-X-even}.
\end{proof}
\indent First, we recall some definitions and results of Appendix A of \cite{MR}. For a cyclic extension $L/K$ of degree $p^n$ in $F$ we define $\Rc_L:=R_L\otimes\Z_p$, where $R_L$ is as in \S\ref{SelmerTate}, and consider $\Rc_L$ as a $G_K$-module by letting $G_K$ act trivially. Let $\zeta$ be a primitive $p^n$ root of unity and denote $\iota$ for the involution of $R_L$ induced by $\zeta\mapsto\zeta^{-1}$, and similarly for $\Rc_L$. Let $\pi:=\zeta-\zeta^{-1}$, which is a generator of the unique prime $\hat{p}$ of $R_L$ above $p$ and of the maximal ideal $\Pp$ of $\Rc_L$.\\
\indent For $W$ an $\Rc_L$-module and $B$ a $\Z_p$-module, a pairing $\ip{~}{~}:W\times W\to B$ is $\iota$-\emph{adjoint} if for each $r\in\Rc_L$ and $x,y\in W$, $\ip{rx}{y}=\ip{x}{r^\iota y}.$ Also, a pairing $\ip{~}{~}:W\times W\to \Rc_L\otimes_{\Z_p}B$ is $\Rc_L$-\emph{semilinear} if for each $r\in\Rc_L$ and $x,y\in W$
		$$\ip{rx}{y}=r\ip{x}{y}=\ip{x}{r^\iota y},$$
and is \emph{skew-Hermitian} if it is $\Rc_L$-semilinear and $\ip{y}{x}=-\ip{x}{y}^{\iota\otimes 1}$.\\
\indent Mazur and Rubin construct a map $\tau:\Rc_L\to\Z_p$ such that composition with $\tau\otimes 1:\Rc_L\otimes_{\Z_p}B\to B$ gives a bijection (Lemma A.3 and Proposition A.4 of \cite{MR}) between the set of $\Rc_L$-semilinear pairings $W\times W\to \Rc_L\otimes_{\Z_p}B$ and the set of $\iota$-adjoint pairings $W\times W\to B$. Also, if $\ip{~}{~}_{\Rc_L}$ corresponds to $\ip{~}{~}_{\Z_p}$ then $\ip{~}{~}_{\Rc_L}$ is perfect (resp. $G_K$-equivariant) if and only if $\ip{~}{~}_{\Z_p}$ is perfect (resp. $G_K$-equivariant).\\
\begin{definition}[Definition A.5 of \cite{MR}]
\label{MRA.5}
	Let $p^n=[L:K]$. Define two pairings: $f:\Ic_L\times\Ic_L\to R_L$ by
		$$f(\alpha,\beta):=\pi^{-2p^{n-1}}\alpha\beta^\iota,$$
and $\ip{~}{~}_{\Rc_L}:=f\otimes e_{p,\lambda}$ on $T_p(A_L)=\Ic_L\otimes T_p(X)$ by
		\begin{equation}
		\label{pairA.5}	
			\ip{\alpha\otimes x}{\beta\otimes y}:=(\pi^{-2p^{n-1}}\alpha\beta^\iota)\otimes\weilpairpol{p}{x,y}
					\in{\Rc_L}\otimes_{\Z_p}\Z_p(1).
		\end{equation}
\end{definition}
\indent In Theorem A.12 of \cite{MR}, Mazur and Rubin use the pairing \eqref{pairA.5} and arguments of Flach \cite{Flach} to obtain a perfect, skew-Hermitian, $\Gal{K}{k}$-equivariant pairing $[~,~]_{\Rc_L}$ on
		$$\Sha(A_L/K)_{/\text{div}}:=\Sha(A_L/K)/\Sha(A_L/K)_\text{div},$$
taking values in $D_p:=\Rc_L\otimes_{\Z_p}\Q_p/\Z_p$. Using Flach's arguments, we can also obtain the classical Cassels-Tate pairing on $\Sha(X/K)_{/\text{div}}$ from the Weil pairing on $X[p]$. We first show that these pairings satisfy $[sx,y]=[x,s^\dagger y],$ for each $s\in\Oc$. 
\begin{prop}
\label{Radjoint}
	Suppose $Y/k$ is an abelian variety with an action of $\Oc$ and $B=\Q_p/\Z_p$ or $B=D_p$. If $\ip{~}{~}:T_p(Y)\times T_p(Y)\to B$ induces (via Flach's construction) $[~,~]$ on $\Sha(Y/K)_{/\text{div}}$ and $\ip{sx}{y}=\ip{x}{s^\dagger y}$ for all $s\in\Oc$, then $[sx,y]=[x,s^\dagger y]$ for all $s\in\Oc$.
\end{prop}
\begin{proof}
	We recall the construction of $[~,~]$ from p.116 of \cite{Flach}. Let $V_p(Y)=T_p(Y)\otimes\Q$. From $x,x'\in\Selmer{p^\infty}{Y}{K}$, we obtain cocylces $\alpha,\alpha'\in\Cocyc{1}{K}{Y[p^\infty]}$. From the exact diagram
		$$\xymatrix@R=30pt@C=30pt
			{	 & \Coch{1}{K}{V_p(Y)}\ar[r]\ar[d]^d & \Coch{1}{K}{Y[p^\infty]}\ar[d]^d\ar[r] & 0\\
				\Coch{2}{K}{T_p(Y)}\ar[r] & \Coch{2}{K}{V_p(Y)}\ar[r] & \Coch{2}{K}{Y[p^\infty]} & 
		  }
		$$
we see that $\alpha$ and $\alpha'$ can be lifted to $\beta,\beta'\in\Coch{1}{K}{V_p(Y)}$, and we have $d\beta, d\beta'\in\Coch{2}{K}{T_p(Y)}$. The pairing $\ip{~}{~}$ induces a cup-product $\cup$
		$$\Coch{i}{K}{V_p(Y)}\times\Coch{j}{K}{V_p(Y)}\stackrel{\cup}{\longrightarrow}\Coch{i+j}{K}{B}.$$	
Since $\Coh{3}{K}{B}=0$, there is some $\epsilon\in\Coch{2}{K}{B}$ such that $d\beta \cup \beta'=d\epsilon$. Since $\alpha'$ represents $x\in\Selmer{p^\infty}{Y}{K}$, $\res{v}(\alpha')$ is the image of some cocycle $\beta'_v\in\Cocyc{1}{K_v}{V_p(Y)}$. Define
		$$\gamma_v:=\res{v}(\beta)\cup\beta'_v-\res{v}(\epsilon)\in\Coch{2}{K_v}{B},$$
and then $[x,x']:=\sum_v\invmap{v}(\gamma_v).$\\
\indent Just as in Proposition \ref{Tatepairadjoint}, the cup-product $\cup$ satisfies an $\Oc$-adjoint property, so
	$$\begin{array}{rll}
			d(s\beta) \cup \beta'= & s(d\beta) \cup \beta' & \\
			= & d\beta \cup s^\dagger\beta',
		\end{array}$$
giving the same $\epsilon$ for both pairs $(sx,x')$ and $(x,s^\dagger x')$. Also,
	$$\res{v}(s\beta) \cup \beta'_v=s(\res{v}(\beta)) \cup \beta'_v=\res{v}(\beta) \cup s^\dagger\beta'_v.$$
Thus the pairs $(sx,x')$ and $(x,s^\dagger x')$ define the same $\gamma_v$, for each $v$, and so $[sx,x']=[x,s^\dagger x']$.
\end{proof}
\begin{cor}
\label{Radjointcor}
	If $[~,~]$ is obtained from $e_{p,\lambda}$ or $\ip{~}{~}_{\Rc_L}$, then $[sx,y]=[x,s^\dagger y]$ for all $s\in\Oc$.
\end{cor}
\begin{proof}
	We have already seen that $\weilpairpol{p}{sx,y}=\weilpairpol{p}{x,s^\dagger y}$. By definition, the $\Oc$-action on $\Ic_L\otimes T_p(X)$ is $s(\alpha\otimes x)=\alpha\otimes(sx)$. Therefore,
		$$\begin{array}{rl}
				\ip{s(\alpha\otimes x)}{\beta\otimes y}= & \ip{\alpha\otimes(sx)}{\beta\otimes y}\\
				 = & (\pi^{-2p^{n-1}}\alpha\beta^\iota)\otimes\weilpairpol{p}{sx,y}\\
				 = & (\pi^{-2p^{n-1}}\alpha\beta^\iota)\otimes\weilpairpol{p}{x,s^\dagger y}\\
				 = & \ip{\alpha\otimes x}{\beta\otimes (s^\dagger y)}=\ip{\alpha\otimes x}{s^\dagger(\beta\otimes y)},
			\end{array}$$
and Proposition \ref{Radjoint} gives the claim.
\end{proof}
\begin{prop}
\label{CTpairingc}
	Let $[~,~]$ denote the Cassels-Tate pairing
		$$\Sha(X/K)_{/\text{div}}\times\Sha(X/K)_{/\text{div}}\to\Q_p/\Z_p.$$
Then $[c^*(x),x']=[x,c^*(x')]$.
\end{prop}
\begin{proof}
	Recall that $e_{p,\lambda}$ is $G_k$-equivariant. We keep the notation in the proof of Proposition \ref{Radjoint}. Specifically, let $B=\Q_p/\Z_p$ and let $x,x'\in\Selmer{p^\infty}{X}{K}$. Just as in Proposition \ref{galeqweil} the $G_k$-equivariance of $e_{p,\lambda}$ implies, for any cochains $\omega$, $\omega'$,
		\begin{equation}
		\label{galeqcupprod}
			c^*(c^*(\omega)\cup\omega')=\omega\cup c^*(\omega').
		\end{equation}
Let the pair $c^*(\beta),\beta'$ (resp. $\beta,c^*(\beta')$) define $\epsilon\in\Coch{2}{K}{B}$ and $\gamma_v\in\Coch{2}{K_v}{B}$ (resp. $\epsilon'$, $\gamma'_v$) as in Proposition \ref{Radjoint}. Property \eqref{galeqcupprod} then implies that $c^*(\epsilon)=\epsilon'$. From $c^*\circ\res{v}=\res{v^c}\circ c^*$, we obtain
	$$\begin{array}{rl}
			\gamma_v'= & \res{v}(\beta)\cup c^*(\beta'_{v^c})-\res{v}(c^*(\epsilon))\\
			= & c^*(\res{v^c}(c^*(\beta))\cup\beta'_{v^c}-\res{v^c}(\epsilon))\\
			= & c^*(\gamma_{v^c}),
	  \end{array}$$
and so $\sum_v \invmap{v}(\gamma_v')=\sum_v \invmap{v}\circ c^*(\gamma_{v^c})=\sum_v \invmap{v^c}(\gamma_{v^c}).$ Thus, we conclude that $[x,c^*(x')]=[c^*(x),x'].$
\end{proof}
\begin{remark}
\label{CTpairingc-remark}
	The proposition also follows from Theorem A.12 of \cite{MR}. In particular, Mazur and Rubin show that the $G_k$-equivariance of $e_{p,\lambda}$ implies $\Gal{K}{k}$-equivariance of $[~,~]$, and $\Gal{K}{k}$ acts trivially on $\Q_p/\Z_p$.
\end{remark}
The following proposition shows that $d=\rk{R}\Sha(X/K)_{/\text{div}}[p]$ is even. Theorem 1 of \cite{Flach} shows that $\Sha(X/K)_{/\text{div}}$ is finite, and in particular it is a finite $p$-group. Thus, for some $t\geq 1$
	$$\Sha(X/K)_{/\text{div}}=\Sha(X/K)_{/\text{div}}[p^t]=\oplus_i \Sha(X/K)_{/\text{div}}[\pp_i^t].$$
\begin{prop}
\label{Sha-X-even}
	$d=\rk{R}(\TS{X}{K}[p^\infty]/\TS{X}{K}[p^\infty]_{\text{div}})[p]$ is even.
\end{prop}
\begin{proof}
	From Corollary \ref{Radjointcor} and Proposition \ref{CTpairingc} the pairing $[~,~]$ on $\Sha(X/K)_{/\text{div}}$ satisfies $[sx,x']=[x,s^\dagger x']$ and $[c^*(x),x']=[x,c^*(x')]$ for all $s\in\Oc$ and $x,x'\in\Sha(X/K)_{/\text{div}}$. Define $[~,~]'$ by $[x,y]':=[x,c^*(y)]$ as in Lemma \ref{gen1.3ii}, obtaining a non-degenerate, skew-symmetric, $\Z_p$-bilinear pairing on $\Sha(X/K)_{/\text{div}}$ with $[sx,y]'=[x,sy]'$ for all $s\in\Oc$ and $x,y\in\Sha(X/K)_{/\text{div}}$. Since $\Sha(X/K)_{/\text{div}}$ is finite, Proposition \ref{skewsymmprop} (with $A=\Oc\otimes\Z_p$) then shows that $d$ is even.
\end{proof}
We now provide the analogous statement to Proposition \ref{MR2.1congruence} for $A_L$. Previously, we noted that the twist $A_L$ is defined over $K$, but in fact it is essential that $A_L$ have a model over $k$ in order to apply Theorem A.12 of \cite{MR}. Again, the results of Appendix A of \cite{MR} (Definition A.8 and on, or alternatively \cite[\S 6]{MRS}) allow us to consider $A_L$ defined over $k$. Combining Propositions \ref{MR2.1congruence} and \ref{genMR6.3} in Theorem \ref{genMR6.4} below proves a generalization of Theorem 6.4 of \cite{MR}. Recall $\Rc_L=R_L\otimes\Z_p$.
\begin{prop}
\label{genMR6.3}
	$$\crk{\Oc\otimes\Rc_L}\Selmer{p^\infty}{A_L}{K}\equiv\rk{R}\Selmer{\hat{p}}{A_L}{K}-\rk{R}X(K)[p]\mod{2}.$$
\end{prop}
\begin{proof}
	The proof is the same as Proposition \ref{MR2.1congruence}, using Proposition \ref{gen4.1} to identify $A_L(K)[\hat{p}]$ with $E(K)[p]$, and seeing that $d=\rk{R}\Sha(A_L/K)_{/\text{div}}[\hat{p}]$ is even as follows. Theorem 1 of \cite{Flach} shows $M=\Sha(A_L/K)_{/\text{div}}$ is an $\Oc\otimes\Rc_L$-module of finite cardinality. Since
		$$\Oc\otimes\Rc_L=\Oc\otimes(\Z_p\otimes R_L)=(\Oc\otimes\Z_p)\otimes R_L,$$
we have $M=\oplus_j(M\otimes\Oc'_{\hat{\pp}_j})$, where $\Oc'_{\hat{\pp}_j}=\Oc_{\pp_j}\otimes R_L$. As noted above, Theorem A.12 of \cite{MR} produces a perfect, skew-Hermitian, $\Gal{K}{k}$-equivariant pairing $[~,~]$. Defining $[x,y]'=[x,c^*(y)]$ as before gives a non-degenerate, skew-symmetric, $\Rc_L$-bilinear pairing with $[sx,y]'=[x,sy]'$ for all $s\in\Oc$. We can therefore apply Proposition \ref{skewsymmprop} (with $A=\Oc\otimes\Rc_L$) to see $d$ is even.
\end{proof}

\section{Main Results}
\label{primeresults}
We are now in a position to define and make use of the arithmetic local constants for our abelian variety $X$. Recall $R=\Oc/p\Oc$, where $\Oc\subset\End{K}{X}$. Also, recall that for each cyclic $L/K$, we have a twist $A_L$ of $X$ and rings $R_L$ (see \S\ref{SelmerTate}) and $\Rc_L=R_L\otimes\Z_p$.
\begin{definition}
\label{genMR4.5}
As in Definition 4.5 of \cite{MR}, for each cyclic $L/K$ contained in $F$, we define the arithmetic local constant $\delta_v:=\delta(v,X,L/K)$ by
	$$\delta_v:=\rk{R}(\Sel{\Xc}{K_v}{X[p]}/\Sel{\Xc\cap\Ac}{K_v}{X[p]})\mod{2}.$$
\end{definition}
\begin{theorem}
\label{genMR6.4}
	For $\Sk_L$ as in \S\ref{notation},
	$$\crk{\Oc\otimes\Z_p}\Selmer{p^\infty}{X}{K}-\crk{\Oc\otimes\Rc_L}\Selmer{p^\infty}{A_L}{K}
			\equiv\sum_{v\in\Sk_L} \delta_v\mod{2}.$$
\end{theorem}
\begin{proof}
	First, Lemma \ref{X=A} and Theorem \ref{gen1.4} give
		$$\rk{R}\Selmer{p}{X}{K}-\rk{R}\Selmer{\hat{p}}{A_L}{K}\equiv\sum_{v\in\Sk_L}\delta_v\mod{2}.$$
The claim then follows from this, Proposition \ref{MR2.1congruence} and Proposition \ref{genMR6.3}.
\end{proof}
Corollary 5.3 of \cite{MR} shows that in the elliptic curve case, $\delta_v$ can be computed via a completely local formulae, and the same arguments apply in our more general setting. For $v$ a prime of $K$ and $w$ a prime of $L$ above $v$, if $L_w\neq K_v$, let $L'_w$ be the unique subfield of $L_w$ containing $K_v$ with $[L_w:L'_w]=p$, and otherwise let $L'_w:=L_w=K_v$. Proposition 5.2 of \cite{MR} provides an $\Oc$-module isomorphism
	\begin{equation}
	\label{MR5.2}
		\Sel{\Xc\cap\Ac}{K_v}{X[p]}\cong(X(K_v)\cap\Norm{L_w}{L'_w}X(L_w))/pX(K_v).
	\end{equation}
\begin{prop}[Corollary 5.3 of \cite{MR}]
\label{genMR5.3}
	For every prime $v$ of $K$, \eqref{MR5.2} implies
		$$\delta_v\equiv\rk{R}X(K_v)/(X(K_v)\cap\Norm{L_w}{L'_w}X(L_w))\mod{2}.$$
\end{prop}
\begin{cor}
\label{set-c-ramified}
	Let $\Sk^c_L$ be the set of primes $v$ of $K$ such that $v$ ramifies in $L/K$ and $v^c=v$. Then 
		$$\crk{\Oc\otimes\Z_p}\Selmer{p^\infty}{X}{K}-\crk{\Oc\otimes\Rc_L}\Selmer{p^\infty}{A_L}{K}
				\equiv\sum_{v\in\Sk^c_L} \delta_v\mod{2}.$$
\end{cor}
\begin{proof}
	The arguments are as in the proof of Theorem 7.1 of \cite{MR}. If $v\not\in\Sk^c_L$ then $v^c\neq v$ or $v$ is unramified in $L/K$. If $v^c\neq v$ then Lemma 5.1 of \cite{MR} shows that $\delta_v+\delta_{v^c}\equiv 0$. If $v^c=v$ and $v$ is unramified then, Lemma 6.5 of \cite{MR} shows that $v$ splits completely in $L/K$ and hence $\Norm{L_w}{L'_w}$ is surjective. Using Proposition \ref{genMR5.3}, we see that $\delta_v\equiv 0$.
\end{proof}
The following is a first example of a class of abelian varieties for which Proposition \ref{genMR6.4} can be used to produce a lower bound for the growth in $p$-Selmer ($\Oc\otimes\Z_p$)-rank.
\begin{cor}
\label{abvarfirstexample}
	Suppose that for every $v\in\Sk^c_F$, we have $v\mid p$ and $X$ has good ordinary, non-anomalous reduction at $v$. If $\crk{\Oc\otimes\Z_p}\Selmer{p^\infty}{X}{K}$ is odd then
		$$\crk{\Oc\otimes\Z_p}\Selmer{p^\infty}{X}{F}\geq ([F:K],\ldots,[F:K]).$$
\end{cor}
\begin{proof}
	Suppose $L/K$ is a cyclic extension contained in $F$. Theorem \ref{genMR6.4} and Corollary \ref{set-c-ramified} show that we need only see that $\delta_v=0$ for all $v\in\Sk^c_F$. Since $v\in\Sk^c_F$, we have $v$ is totally ramified in $L_w/K_v$ by Lemma 6.5 of \cite{MR}.\\
	\indent The assumptions that $v\mid p$ and that $X$ has good ordinary, non-anomalous reduction at $v$ allow us to apply the arguments of Appendix B of \cite{MR} to see $\delta_v=0$. The key ingredients therein are, firstly, the diagram on page 239 of \cite{LubRos}, which applies to abelian varieties of any dimension. Secondly, non-anomalous reduction guarantees the relevant norm maps are surjective.\\
	\indent Now, for each cyclic $L$ in $F$, we have
	$$\crk{\Oc\otimes\Z_p}\Selmer{p^\infty}{X}{K}\equiv\crk{\Oc\otimes\Rc_L}\Selmer{p^\infty}{A_L}{K}\mod{2},$$
and by our hypotheses, the left-hand side is odd. As in Theorem 7.1 of \cite{MR}, the Pontrjagin dual $\Sc_p(X/F)$ of $\Selmer{p^\infty}{X}{F}$ (see for example \cite[\S 3]{MR}) decomposes as
	$$\Sc_p(X/F)\cong\oplus_L\Sc_p(A_L/K),$$
with each $\Sc_p(A_L/K)$ a $\bK[\Gal{F}{K}]_L\otimes\Q_p$-module (see \S 3 and Remark \ref{O-twist-vs-Z-twist}), and we have just seen each has odd dimension. From $\bK[\Gal{F}{K}]\cong\oplus_L\bK[\Gal{F}{K}]_L$, we see that $\Sc_p(X/F)$ contains a submodule isomorphic to 
	$$\bK_p[\Gal{F}{K}]\cong\oplus_L (\bK[\Gal{F}{K}]_L\otimes\Q_p),$$
and the claim follows.
\end{proof}

\subsection{Composite Dihedral Extensions}
\label{compositeresults}
We now consider an abelian extension $F/K$ of odd degree $[F:K]=m$, and a cyclic extension $L/K$ inside $F$. To ease notation, we fix some ordering of the primes in $[L:K]=\prod_i p_i^{e_i}$, where $e_i>0$ for each $i$. For such $L/K$ in $F$ and each $i$, there exists a $p_i$-power subextension $M_i/K$ such that $L/M_i$ is of degree prime to $p_i$.\\
\indent By Proposition 5.10 of \cite{MRS}, if $M$ and $M'$ are cyclic extensions of $K$ inside $L$ with $[M:K]$ and $[M':K]$ coprime and $L=MM'$, then the twist $A_L$ of $X$ with respect to $L/K$ may also be realized as a twist of $A_M$, i.e. $A_L\cong(A_M)_{M'}$. Thus, if we want to compare $A_L$ and $X$, it suffices to compare $X$ with $A_M$, and also $A_M$ with $(A_M)_{M'}$. As in the paragraph preceeding Proposition \ref{genMR6.3}, we consider $A_M$ and $(A_M)_{M'}$ as defined over $k$.\\
\indent In order to inductively apply Theorem \ref{genMR6.4} (see Theorem \ref{compositedihedral} below), we assume the following conjecture.
%
\begin{conjecture}
\label{rankindepofl}
	Suppose $p$ is a prime, $Y/L$ is an abelian variety, $B\subset\End{L}{Y}$ is an integral domain, and $\qp$ and $\qp'$ are primes of $B$ above $p$. Then 
		\begin{enumerate}
			\item $\crk{B\otimes\Z_{p}}\Selmer{p^\infty}{Y}{L}$ is independent of $p$,
			\item $\crk{B_{\qp}}\Selmer{p^\infty}{Y}{L}\otimes B_{\qp}=\crk{B_{\qp'}}\Selmer{p^\infty}{Y}{L}\otimes B_{\qp'}$,
		\end{enumerate}
\end{conjecture}
\begin{remark}
\label{rankindepofl-remark}
	Both parts of the conjecture follow from the Shafarevich-Tate Conjecture. Indeed, when $\#\Sha(Y/L)<\infty$, \eqref{limitselmerexactsequence} implies
		$$\crk{B\otimes\Z_{p}}\Selmer{p^\infty}{Y}{L}=\rk{B\otimes\Z_{p}}(Y(L)\otimes\Q_{p}/\Z_{p})=(\ldots,\rk{B} Y(L),\ldots).$$
Each entry in the tuple is identical, giving (2), and independent of $p$, giving (1).
\end{remark}
\indent For the remainder, we let $F/K$ be as at the beginning of \S\ref{compositeresults} with $F/k$ dihedral, $X/k$ and $\Oc\subset\End{K}{X}$ as in the previous sections (see \S\ref{notation}), and assume that each prime dividing $[F:K]$ is unramified in $\Oc$. For Theorem \ref{compositedihedral} below, we also fix a cyclic extension $L/K$ in $F$.\\
\indent For each $M/K$ in $L$, let $R_M$ denote the maximal order in $\Q[\Gal{F}{K}]_M$ (as in \S\ref{SelmerTate} for $M=L$) and $\Oc_M=\Oc\otimes R_M$.\footnotemark ~Recall $c$ is the non-trivial element of $\Gal{K}{k}$. Let (as in Corollary \ref{set-c-ramified})
\footnotetext{We note that $\Oc_L\otimes\Z_p\cong\Oc\otimes\Rc_L$, with the latter as in Theorem \ref{genMR6.4}. The new notation is more convenient for dealing with more than one prime.}
	$$\Sk^c_M:=\braces{\text{primes $v$ of $K$ : $v^c=v$ and $v$ ramifies in $M/K$}}.$$
Set $M_0=K$ and for each $i>0$ set $M_i\subset L$ to be a $p_i$-extension of $K$ such that $p_i\nmid[L:M_i]$.\\
\indent Using Conjecture \ref{rankindepofl} (2), for any $p$, the tuple defining $\crk{B\otimes\Z_p}\Selmer{p^\infty}{Y}{L}$ may be thought of as a single value, so we define $r_p(Y/L,B)\in\Z$ by
	$$r_p(Y/L,B):=\crk{B_\qp}\Selmer{p^\infty}{Y}{L}\otimes B_\qp,$$
where $\qp$ is some prime of $B$ above $p$. In turn, one may interpret the right-hand side of Theorem \ref{genMR6.4} as a single value, so we define $\delta(X,L/K)\in\Z/2\Z$ as
	$$\delta(X,L/K):=\text{the first component of}~\left(\sum_{v\in\Sk_L} \delta(v,X,L/K)\right)\mod{2}.$$
\begin{remark}
\label{delta-constant-parity}
	We emphasize that the \emph{sum} of the local constants $\delta(v,X,L/K)$, for fixed $X$ and $L/K$, has constant parity across components, by Conjecture \ref{rankindepofl} (2) and Theorem \ref{genMR6.4}. It would be interesting to determine under what conditions one can prove that the \emph{individual} $\delta(v,X,L/K)$ have constant parity across components.
\end{remark}
\begin{theorem}
\label{compositedihedral}
	Assume Conjecture \ref{rankindepofl}. For $K=M_0,M_1,\cdots,L$ as above and $p$ a prime dividing $[L:K]$,
		$$r_p(A_L /K,\Oc_L)-r_p(X/K,\Oc)\equiv\sum_{i\geq 1}\delta(A_{M_{i-1}},M_i/K)\mod{2}.$$
\end{theorem}
\begin{proof}
	Without loss of generality we may assume $p=p_1$. We proceed by induction on the number $j$ of primes dividing $[L:K]$, and the case $j=1$ is that of Theorem \ref{genMR6.4}. Suppose now that $j>1$, and let $M=M_1$ and let $M'$ correspond to the compositum of the $M_i$ for $1<i\leq j$. Recall from the discussion above that Proposition 5.10 of \cite{MRS} shows $A_L\cong(A_M)_{M'}.$ Arguments of Howe \cite[\S 2]{Howe} show that $A_M$ has a polarization degree of $p^2$, in particular prime to $[L:M]$, and so we can apply Theorem \ref{genMR6.4} in $L/M$ with $A_M$ playing the role of $X$. For $p'$ any prime dividing $[L:M]$, by induction we have
		$$r_{p'}(A_L,\Oc_L)-r_{p'}(A_M,\Oc_M)\equiv \sum_{i\geq 2}\delta(A_{M_{i-1}},M_i/M)\mod{2}.$$
Using Conjecture \ref{rankindepofl} (1), we have
		$$r_p(Y/K,B)\equiv r_{p'}(Y/K,B)\mod{2},$$
for $Y=X$, $A_M$, $A_L$, and $B=\Oc$, $\Oc_M$, $\Oc_L$, respectively, and hence
	$$\begin{array}{rl}
			r_p(A_L/K,\Oc_L) - r_p(X/K,\Oc)\equiv & r_{p'}(A_L/K,\Oc_L) - r_{p'}(A_M/K,\Oc_M)\vspace{.05in}\\
																						& \hspace{.15in}+~r_{p}(A_M/K,\Oc_M) - r_{p}(X/K,\Oc)\vspace{.15in}\\
			\equiv & \sum_{i\geq 2}\delta(A_{M_{i-1}},M_i/M)\vspace{.05in}\\
			       & \hspace{.35in}+~\delta(X,M/K)\vspace{.15in}\\
			\equiv & \sum_{i\geq 1}\delta(A_{M_{i-1}},M_i/K)\mod{2}.
		\end{array}$$
We are able to restrict the primes $v$ in the preceeding sums to those in $\Sk^c_{M_i}$ just as in Corollary \ref{set-c-ramified}.
\end{proof}
As in Corollary \ref{abvarfirstexample}, the following is a first example of a setting in which Theorem \ref{compositedihedral} can be used to provide a lower bound for growth in the rank of $E$ (i.e. when $X=E$ is an elliptic curve).
\begin{cor}
\label{compositefirstexample}
	Let $E/k$ be an elliptic curve, $\bK\not\subset k$, and assume $\#\Sha(E/F)<\infty$. For each cyclic $L/K$ let $M_{L,i}\subset L$ be as in the paragraphs preceeding Theorem \ref{compositedihedral}. Suppose that for every prime $v$ of $K$,
		\begin{enumerate}
			\item if $v=v^c$ then $v$ is unramified in $M_{L,i}/K$ for every $L$ and each $i\geq 2$,
			\item if $v=v^c$ and $v$ ramifies in $M_{L,1}/K$ then $v\nmid p_1$ and $E$ has good reduction at $v$.
		\end{enumerate}
	Let $m$ be the number of primes $v$ satisfying (2). If $\rk{\Oc}E(K)+m$ is odd, then $\rk{\Oc}E(F)\geq [F:K].$
\end{cor}
\begin{proof} Fix a cyclic extension $L/K$ inside $F$, and set $M_i=M_{L,i}$. From $\#\Sha(E/F)<\infty$ we have (e.g.) $\rk{\Oc}E(K)=r_p(E/K,\Oc)$ and Conjecture \ref{rankindepofl}, so we are in the situation of Theorem \ref{compositedihedral}. As in Corollary \ref{set-c-ramified}, if $v$ is unramified or $v\neq v^c$ then $\delta(v,A_{m_{i-1}},M_i/K)\equiv 0$ or
	$$\delta(v,A_{m_{i-1}},M_i/K)+\delta(v^c,A_{m_{i-1}},M_i/K)\equiv 0,$$ 
respectively, for every $i\geq 1$. For $v=v^c$, condition (1) gives $\delta(v,A_{m_{i-1}},M_i/K)\equiv 0$, for every $i\geq 2$. Thus $\delta(E,M_i/K)\equiv 0$ for $i\geq 2$. By Theorem 2.8 of \cite{ChLi}, condition (2) along with $\bK\not\subset k$ gives $\delta(v,E,M_1/K)\equiv (1,1)$, and so $\delta(E,M_1/K)\equiv m$.\\
\indent Using Theorem \ref{compositedihedral}, we combine the calculations to see that
	$$r_p(A_L /K,\Oc_L)\equiv r_p(E/K,\Oc)+m\mod{2}.$$
By assumption, this is forces $r_p(A_L /K,\Oc_L)$ to be odd and hence at least 1. The claim then follows just as in Corollary \ref{abvarfirstexample}.
\end{proof}

\subsection*{Acknowledgements}
This material is based upon work supported by the National Science Foundation under grant DMS-0457481. The author would like to thank Karl Rubin for his many helpful conversations on this material, and thank Karl Rubin and Jan Nekov\'{a}\v{r} for comments on initial drafts of this paper.


\bibliographystyle{plain}
\bibliography{research}


\end{document}